\documentclass[a4paper,12pt]{amsart}
\usepackage{amssymb, amsmath, mathrsfs}
\usepackage{mathtools}
\usepackage{enumitem}
\mathtoolsset{showonlyrefs}
\usepackage[inline, final]{showlabels}
\usepackage{tikz-cd}
\usepackage{enumitem}
\usepackage{wrapfig}
\usepackage{tikz}
\usetikzlibrary{arrows}

\pgfdeclarelayer{background}
\pgfsetlayers{background,main}
\definecolor{col1}{RGB}{236,126,112}
\definecolor{col2}{RGB}{177,184,180}
\definecolor{col3}{RGB}{38,73,68}

\usepackage[disable]{todonotes}
\usepackage[
 pdfauthor={Asbjorn C. Nordentoft, Yiannis N. Petridis and Morten S. Risager},
 pdftitle={Bounds on  shifted convolution sums  for Hecke eigenforms},
 pdfkeywords={},
 pdfproducer={LaTeX2e with hyperref},
 pdfcreator={Pdflatex},
 pdfdisplaydoctitle =true]
 {hyperref}
\usepackage{pdfsync}

\newtheorem{thm}{Theorem}[section]
\newtheorem{prop}[thm]{Proposition}
\newtheorem{lem}[thm]{Lemma}

\newtheorem*{rem*}{Remark}

\theoremstyle{definition}

\newtheorem{rem}{Remark}

\def \e {{\varepsilon}}

\def \R {{\mathbb R}}

\def \N {{\mathbb N}}

\def \C {{\mathbb C}}
\def \g {\gamma}
\def \G {\Gamma}
\def \Z {\mathbb{Z}}

\def \slz  {{\hbox{SL}_2( {\mathbb Z})} }

\def \GinfmodG {{\Gamma_{\!\infty}\!\setminus\Gamma}}
\def \supp {{\rm supp\,} }

\newcommand{\abs}[1]{\left\lvert #1 \right\rvert}
\newcommand{\norm}[1]{\left\lVert #1 \right\rVert}

\newcommand{\inprod}[2]{\left \langle #1,#2 \right\rangle}
\DeclareMathOperator*{\res}{Res}

\providecommand{\sym}{\operatorname{sym}}

\title[Shifted convolution sums]{Bounds on  shifted convolution sums  for Hecke eigenforms}
\author[A. Nordentoft]{Asbj\o rn C. Nordentoft}
\address{Mathematical Institute of the University of Bonn, Endenicher Allee 60, Bonn 53115, Germany}
\email{acnordentoft@outlook.com}

\author[Y. Petridis]{Yiannis N. Petridis}
\address{Department of Mathematics, University College London, Gower Street, London WC1E 6BT, United Kingdom}
\email{i.petridis@ucl.ac.uk}

\author[M. Risager]{Morten S. Risager}
\address{Department of Mathematical
  Sciences, University of Copenhagen, Universitetsparken 5, 2100
  Copenhagen \O , Denmark}
\email{risager@math.ku.dk}
\thanks{The research of Asbjørn C. Nordentoft and Morten S. Risager was supported by the Grant DFF-7014-00060B from Independent Research Fund Denmark}

\date{\today}

\begin{document}
\begin{abstract} Shifted convolution sums play a prominent r\^ole in analytic number theory. Here these sums  are considered in the context of holomorphic Hecke eigenforms.  We investigate pointwise bounds, mean-square bounds consistent with the optimal conjectural bound, and  find asymptotics on average  for their variance.
\end{abstract}

\maketitle
\section{Introduction}
Sums of the form
\begin{equation}A(X,h)=\sum_{n\leq X}\lambda(n)\lambda(n+h)\label{general}\end{equation}
play an important r\^ole in analytic number theory, especially when $\lambda$ is of arithmetic significance, see e.g. \cite{Ingham:1927,  Selberg:1965a, DeshouillersIwaniec:1982, Good:1983a, IvicMotohashi:1995, ConreyGonek:2001,BlomerHarcos:2008a, Holowinsky:2009} and the references therein. The case where $\lambda(n)$ is the $n$th Hecke eigenvalue of an automorphic object is maybe the most interesting, and in this case the  above sum is called a \emph{shifted convolution sum} or sometimes a \emph{generalized additive divisor sum}. Here are some examples:
\begin{enumerate}[label=(\alph*)]
\item \label{divisor-sums} When $\lambda(n)$ equals the divisor function $d(n)=\sum_{b\vert n} 1$ this is the classical binary additive divisor problem. In this case $\lambda(n)$ are the Hecke eigenvalues (of the derivative in $s$ at $s=1/2$) of the weight 0 non-holomorphic Eisenstein series $E(z,s)$ for the full modular group.  See e.g. \cite[p. 62]{Iwaniec:2002a}.
\item \label{hyp-lattice} When $\lambda(n)$ equals $r_2(n)$ i.e. the number of ways of representing $n$ as the sum of two squares and $h=1$ this is a special case of the hyperbolic circle problem, see e.g. \cite[p. 174]{Iwaniec:2002a}. In this case $\lambda(n)$ is equal to the Hecke eigenvalues of the theta series $\theta_2(z)=\sum_{n\in \Z^2} e^{2\pi i \norm{n}^2 z}$.
\end{enumerate}
In both of the above cases the automorphic object is non-cuspidal and the sum $A(X,h)$ admits an asymptotic expansion with a main term $XP(\log X)$ where $P$ is a polynomial, and an error term which is provably $O_h(X^{2/3+\e})$ (See \cite[Thm. 1]{DeshouillersIwaniec:1982}\cite[Cor. 1 ]{Motohashi:1994a}, \cite[Chapter 12]{Iwaniec:2002a}), and conjecturally $O(X^{1/2+\e}h^\e)$.

In this paper we consider the automorphic object to be a weight $k$, level $1$ holomorphic cuspidal Hecke eigenform $f\in S_k(1)$, normalized such that its Fourier expansion \begin{equation}
f(z)=\sum_{n=1}^\infty\lambda_f(n)n^{\frac{k-1}{2}}e(nz),
\end{equation}  satisfies $\lambda_f(1)=1$. Here $e(z)=e^{2\pi i z} $.  We note that, since the Hecke operators are self-adjoint, $\lambda_f(n)$ is real. In this cuspidal case there is no main term but the provable error term is still of the same size as in the non-cuspidal case, i.e.
\begin{equation}\label{jutila-bound}
  A_f(X,h):=\sum_{n\leq X}\lambda_f(n)\lambda_f(n+h)=O_f(X^{2/3+\e})
\end{equation}
uniformly for $1\leq h\leq X^{2/3}$. This was proved by Jutila \cite[Eq (1.32)]{Jutila:1996a}, and we sketch a variant of his proof in Section \ref{sec:spectral-expansion}. Note that Jutila uses a different normalization on $\lambda_f(n)$. Also in this case we expect the conjectural bound
\begin{equation}\label{conjectural-bound}
  A_f(X,h)=O_f(X^{1/2+\e}h^\e)
\end{equation}
for $h\leq X^{1/2-\e}$.
Note that the implied constant depends on $f$. Probably this implied constant is bounded by a constant times $k^\e$ at least for certain ranges of $k$ vs $X$, but the precise conjectural range is not obvious.

One frequently encounters smooth sums e.g.
\begin{equation}\label{smooth-sum-def}
  A_f^{W}(X,h)=\sum_{n\in\N}\lambda_f(n)\lambda_f(n+h)W\left(\frac{
  n+h/2}{X}\right),
\end{equation}
where $W:\R_+\to\R$ is smooth and supported in a compact interval, e.g.  on $[1,2]$.  Note that we are summing over the range where the average of $n$ and $n+h$ is of size roughly $X$.
It is perhaps not too surprising that the analogue of the conjectural bound \eqref{conjectural-bound} holds for the smooth sum $ A_f^{W}(X,h)$, and we will give a short proof of this well-known fact (see \cite[Eq. (5)]{BlomerHarcos:2008a}).
\begin{prop}\label{smooth-sum} Assume the Ramanujan--Petersson conjecture for Maass forms. Then
\begin{equation}  A_f^{W}(X,h)=O_{f, W}(h^\e X^{1/2+\e})
\end{equation}
for $h=O(X^{1/2-\e})$.
\end{prop}
Without assuming the Ramanujan--Petersson conjecture one gets a slightly weaker bound, but one which holds in a larger range of $h$. See Proposition \ref{unconditional-smooth} below for details, as well as comparison with similar results in the existing literature.

For non-smooth sums we prove the same bound in the mean-square.
\begin{thm}\label{mean-square-theorem} For $h\leq X^{1/2}$ we have
\begin{equation}
\left(\frac{1}{X}\int_{X}^{2X}\abs{A_f(x,h)}^2dx\right)^{1/2}=O_f(h^{1/2}X^{1/2+\e}).
\end{equation}
\end{thm}
Note that for fixed $f, h$ this is consistent with the conjectural bound of $A_f(x,h)$ and for varying $h$ it is better than the trivial bound coming from \eqref{2/3-bound} as long as $h\leq X^{1/2}$.
The result is analogous to a result proved by Faĭziev \cite[Theorem 3]{Faiziev:1985} and independently by Ivi\'c and Motohashi \cite[Cor. 5]{IvicMotohashi:1994a}, who proved the analogous result for the classical additive divisor problem. For the case \ref{hyp-lattice}, i.e. $\lambda (n)=r_2(n)$ the analogous result is due to Chamizo \cite{Chamizo:1996b}.

The main ingredient in proving Proposition \ref{smooth-sum} and Theorem \ref{mean-square-theorem} is the spectral expansion of the relevant generating series.

When we average $A_f^{W_1}(h_1,X)\overline{A_f^{W_2}(h_2,X)}$ over a Hecke basis $H_k$ of $S_k$ (with suitable weights) we find asymptotics as $X,k\to\infty$ at a certain range. These are consistent with $A_f^{W}(X,h)$ being of the order of $X^{1/2}$  in the $X$-aspect when $X\ll k^{1/2-\e}$. More precisely we have the following:

Let $\tau_1(n)=\sum_{d\mid n} d$ be the sum of divisors function, and let $L(s, \sym^2 f)$ be the symmetric square $L$-function associated to $f$, i.e.
\begin{equation}
   L(s, \sym^2 f)=\zeta(2s)\sum_{n=1}^\infty\frac{\lambda_f(n^2)}{n^s}, \quad\textrm{ for }\Re(s)>1,
\end{equation}and defined on $\C$ by analytic continuation.
For $W_1, W_2 :(0,\infty)\rightarrow \R$ and $h_1,h_1\in \N$ we let
\begin{align} B_{h_1,h_2}(W_1, W_2)=  \tau_1((h_1,h_2))\int_{0 }^\infty W_1(h_1 y) \overline{W_2(h_2y)}dy. \end{align}
Then we show the following result:
\begin{thm} \label{scs-introduction}
Let $W_1, W_2 :(0,\infty)\rightarrow \R$ be smooth with compact support. Then for $X\ll k^{1/2-\e}$ we have
\begin{align}
  \frac{2\pi ^2}{k-1}\sum_{f\in H_k} \frac{A_f^{W_1}(h_1,X)\overline{A_f^{W_2}(h_2,X)}}{L(1, \sym^2 f)}
 =  B_{h_1,h_2}(W_1,W_2) X+ O_{h_i, W_i }(1).
\end{align}

\end{thm}

We recall that the arithmetical weights  $L(1, \sym^2 f)^{-1}$ are relatively well-behaved, in the sense that  $k^{-\e}\ll L(1, \sym^2 f)\ll k^\e$, see \cite{GoldfeldHoffsteinLieman:1994a}. With some work these weights can probably be removed from the Theorem.
The interest in  Theorem \ref{scs-introduction} comes from the study of small scale equidistribution at infinity for Hecke eigenforms, see \cite{NordentoftPetridisRisager:2020a},  where sums with extra average over $k$ are considered.  These results are proved using the  Petersson formula. In the case of Theorem \ref{scs-introduction} the off-diagonal terms arising from the formula are essentially trivial, due to the decay of the $J$-Bessel function.
\section*{Acknowledgements}
We are grateful to the anonymous referee for his/her many insightful comments and suggestions.

\section{Spectral expansions of shifted convolution sums}\label{sec:spectral-expansion} In order to understand the average behavior of $\lambda_f(n)\lambda_f(n+h)$ we consider, following Selberg and others e.g. \cite{Selberg:1965a, Good:1982a, Jutila:1996a, Sarnak:2001a, LauLiuYe:2006a},  the Dirichlet series
\begin{equation}
  D_f(s,h)=\sum_{\substack{n,m=1\\m-n=h}}^\infty\frac{\lambda_f(m)\lambda_f(n)(nm)^{(k-1)/2}}{(n+m+h)^{s+k-1}}.
\end{equation}
This series converges absolutely for $\Re(s)>1$ by standard bounds on the Fourier coefficients, see,  e.g. Deligne's bound \cite{Deligne:1974a}
  \begin{equation}\label{deligne-bound}\abs{\lambda_f(n)}\leq d(n)\end{equation} or the Rankin--Selberg estimate \cite{Rankin:1939b, Selberg:1940}
  \begin{equation}\label{rankin-selberg}
  \sum_{n\leq X}\abs{\lambda_{f}(n)}^2= C_fX+O(X^{3/5}).
\end{equation}  We note that since the Fourier coefficients are real,
\begin{equation}
  \overline{D_{f}(\overline s,h)}=D_f(s,h).
\end{equation}
 In order to understand the analytic properties of $D_f(s,h)$ we consider, following Selberg \cite[(3.10)]{Selberg:1965a}, the Poincar\'e series
\begin{equation}
U_h(z,s)=\sum_{\g\in\GinfmodG}  e(h\g z)\Im(\g z)^s,
\end{equation}
where $h\in \N.$ This is convergent for $\Re(s)>1$ and admits meromorphic continuation on $s\in \C$ as a $\G$-automorphic function. In $\Re(s)\geq 1/2$ its poles are at $s_j=1/2+it_j$, where $1/4+t_j^2$ is an eigenvalue of the automorphic Laplacian, and in $\Re(s)>1/2$ it is square integrable away from its poles. Notice that for $\G=\G(1)$ there are no poles for $\Re(s)>1/2$. By unfolding we find that
\begin{equation}\label{integral-expression-X}
  \inprod{y^k\abs{f}^2}{U_h(\cdot,\overline s)}=(2\pi)^{-(s+k-1)}\Gamma(s+k-1)D_f(s,h) \quad\textrm{ for }\Re(s)>1.
\end{equation}
Since the left-hand side is meromorphic for $s\in \C$, this gives the meromorphic continuation of $D_f(s,h)$ to $s\in \C$.
The function $U_h(z,s)$ is orthogonal to constants by unfolding, and has the spectral expansion
\begin{align}
  U_h(z,s)=\sum_{j=1}^\infty &\inprod{U_h(\cdot,s)}{u_j}u_j(z)\\ \label{expansion}&+\frac{1}{4\pi}\int_\R\inprod{U_h(\cdot, s)}{E (\cdot, 1/2+it)}E(z,1/2+it)dt,
\end{align}
where $\{u_j\}_{j=1}^\infty$ is an orthonormal basis of Maass cusp forms and $E(z,s)$ is the non-holomorphic Eisenstein series at the cusp $i\infty$, see \cite[Thms 4.7 and 7.3]{Iwaniec:2002a}. Again by unfolding we have
\begin{align}
\inprod{U_h(\cdot,s)}{u_j} &= \frac{\overline {a_j(h)}}{(4\pi h)^{(s-1/2)}}\pi^{1/2}\frac{\Gamma(s-1/2+it_j)\Gamma(s-1/2-it_j)}{\Gamma(s)},\\
\inprod{U_h(\cdot,s)}{E(z,1/2+it)} &= \frac{\overline {\varphi (1/2+it,h)}}{(4\pi h)^{(s-1/2)}}\pi^{1/2}\frac{\Gamma(s-1/2+it)\Gamma(s-1/2-it)}{\Gamma(s)},
\end{align}
where
\begin{equation}
  u_j(z)=\sum_{n \neq 0 }a_j(n)\sqrt{y}K_{it_j}(2\pi \abs{n} y)e(nx)
\end{equation}
is the Fourier expansion of $u_j$ at infinity and, similarly, \begin{displaymath}\varphi(s,n)=\frac{2 }{\xi(2s)}\abs{n}^{s-1/2}\sum_{d\vert n}d^{1-2s}\end{displaymath} is the $n$th  Fourier coefficient of $E(z, s)$.

The expansion \eqref{expansion} combined with \eqref{integral-expression-X} gives the following expression, which is useful for understanding the analytic properties of $D_f(s,h)$, cf. \cite[Theorem 1]{Hafner:1983} and \cite[(1.20)]{Jutila:1996a}:
\begin{equation}\label{another-expansion}
h^{s-1/2}D_f(s,h)=\sum_{j=1}^\infty a_j(h) c_j(f){G_k(s,t_j)}+\frac{1}{4\pi}\int_\R \varphi (1/2+it,h)c_{t}(f){G_k(s,t)}dt,
\end{equation}
where
\begin{align}c_{j}(f)&=\inprod{y^k\abs{f}^2}{u_j},\\c_{t}(f)&=\inprod{y^k\abs{f}^2}{ E(\cdot,1/2+it)},\end{align} and
  \begin{equation}
  G_k(s,t)=\frac{(2\pi)^k}{2^s}\frac{\Gamma(s-1/2-it)\Gamma(s-1/2+it)}{\Gamma(s+k-1)\Gamma(s)}.
  \end{equation}
In order to effectively analyze $D_f(s,h)$ using the spectral expansion \eqref{another-expansion} we need the following bound due to Good \cite[Theorem 1]{Good:1981a}:
\begin{equation}\label{Exp-decay}
  \sum_{0< t_j\leq  T}\abs{c_j(f)}^2e^{\pi t_j}+\frac{1}{4\pi}\int_{-T}^T\abs{c_t(f)}^2e^{\pi \abs{t}}dt\ll_f T^{2k}.
\end{equation}
See also \cite[Theorem 1]{Jutila:1996a} for the case of a Maass form $u_{j_0}$ instead of  $y^{k/2}f$. Comparing with Weyl's law \cite[Corollary 11.1]{Iwaniec:2002a}, we observe that  \eqref{Exp-decay} implies that on average $c_j(f)$ is bounded by a constant times $e^{-\pi t_j/2}t_j^{k-1}$.
Note that, by using a triple product identity,  this may be interpreted as the Lindel\"of hypothesis on average for $L(f\times f\times u_j,1/2)$ in the spectral aspect, see \cite{Watson:2002a, Ichino:2008a}.

Concerning the Fourier coefficients we have for $\Gamma=\slz$ the following bound, which follows from Kuznetsov's asymptotic formula \cite[Theorem 6]{Kuznetsov:1980a}, see also \cite[Lemma 2.4]{Motohashi:1997a} and the remarks following its proof:
\begin{equation}\label{FC-bound}
\sum_{T<t_j\leq T+\Delta}{\frac{\abs{a_j(h)}^2}{\cosh(\pi t_j)}}+\frac{1}{4\pi}\int_{T\leq \abs{t}\leq T+\Delta}\frac{\abs{\varphi(1/2+it,h)}^2}{\cosh(\pi t)}dt\ll (T\Delta+h^{1/2})(Th)^\e
\end{equation}
uniform in $h, T$ and $1\leq \Delta\leq T$.  Note that Kuznetsov's result  is stated without the integral but a direct investigation of the integral  gives $\ll T(Th)^\e$.  We note  that again here we are using the arithmeticity of $\G$. For a more general group the best known bound is $T^2+hT$  in a long range $\Delta=T$, see \cite[Eq (9.13)]{Iwaniec:2002a}.
Using these \eqref{Exp-decay} and \eqref{FC-bound} we can prove the following result:
\begin{thm}\label{props-of-D}The function $D_f(s,h)$ admits meromorphic continuation to $s\in \C$. On $\Re(s)\geq 1/2$ the poles are  exactly at $s=1/2\pm it_j$,  with corresponding residues
  \begin{equation}\res_{s=1/2\pm it_j}D_f(s,h)=\frac{a_j(h)c_j(f)}{h^{\pm i t_j}}\res_{s=1/2\pm it_j}G_k(s,t_j).
    \end{equation}
For $1/2+\e\leq \Re(s)\leq 1+\e$ we have the bound
\begin{equation}
  D_f(s,h)\ll_{f,\e} h^{1/2-\Re(s)}\abs{s}^{1/2}(\abs{s}^{1/2}+h^{1/4})\abs{hs}^\e .
\end{equation}
\end{thm}
\begin{proof} The holomorphic continuation to $\Re(s)>1/2$ follows from the spectral expansion \eqref{another-expansion}, as the sum and integral converge absolutely and locally uniformly when $\Re(s)\geq 1/2+\e$. This follows from the following bound, which also proves the claim about the bound on vertical lines:
  Using Cauchy--Schwarz in combination with \eqref{Exp-decay} and \eqref{FC-bound}  we find that
   for $U$ bounded away from zero we have
\begin{align}\label{lets-get-it-right}   \sum_{U<t_j\leq U+\Delta}&\abs{a_j(h)c_j(f)}+\int_{U<\abs{t}\leq U+\Delta}\abs{\varphi(1/2+it,h)c_{t}(f)}\\
  & \ll_{f,\e} \sqrt{U\Delta+h^{1/2}}U^k(Uh)^\e\ll_{f,\e} ({U^{1/2}\Delta^{1/2}+h^{1/4}})U^k(Uh)^\e.
\end{align}
Note that for $s$ in a vertical strip, say $s=\sigma+i\tau$ with $-\infty <a\leq \sigma\leq b<\infty$,  when all arguments of the Gamma functions are bounded away from $0, -1, -2, \ldots $, the Stirling asymptotics gives
\begin{equation}
  G(s,t)\ll_k \frac{e^{-\frac{\pi}{2}(\abs{t+\tau}+\abs{t-\tau})}}{e^{-\pi\abs{\tau}}}\frac{(1+ \abs{\tau-t})^{\sigma-1}(1+\abs{\tau+t})^{\sigma-1}}{(1+\abs{\tau})^{2\sigma+k-2}} .
\end{equation}
 It is convenient to note that the bound is invariant under changing signs on $\tau$ as well as on $t$. The exponential factor is trivial for $\abs{t}<\abs{\tau}$ and equals $e^{-\pi(\abs{t}-\abs{\tau})}$ otherwise.
 We split the sum and integral in  \eqref{another-expansion} according to the various ranges of the spectral parameter $t$. We assume $\abs{\tau}>1$.

\begin{enumerate}[label=(\roman*)]
    \item If the spectral parameter satisfies $\abs{\abs{t}-\abs{\tau}}\leq 1$, then
    $$G_k(s)\ll (1+\abs{s})^{-(k+\sigma-1)},$$
    so combining this with \eqref{lets-get-it-right}  we find that both the corresponding discrete ($\abs{t_j-\abs{\tau}}\leq 1$) and continuous ($\abs{\abs{t}-\abs{\tau}}\leq 1$) contributions  are bounded by \begin{equation*}\abs{s}^{1-\sigma}(\abs{s}^{1/2}+h^{1/4})(h\abs{s})^\e.\end{equation*}

 \item \label{two}If the spectral parameter satisfies $\abs{\abs{t}-\abs{\tau}}> 1$ with $0\leq \abs{t}\leq 2\abs{\tau}$, then
 $$G_k(s,t)\ll (1+\abs{\abs{\tau}-\abs{t}})^{\sigma-1}(1+\abs{s})^{-(k+\sigma-1)}.$$ We further subdivide the sum into $1\leq A<\abs{\abs{t}-\abs{\tau}}\leq 2A$ with $A \leq \abs{\tau}$. Note that this is a sum at height $\ll \abs{\tau}$ of length $\ll A$ and we find that the sum over this range is bounded by a constant times $(1+A)^{\sigma-1}(1+\abs{s})^{-(\sigma-1)}(\tau^{1/2}A^{1/2}+h^{1/4})\abs{sh}^\e$. We therefore bound the  contributions as
 \begin{align}
     \sum_{\substack{0\leq t_j\leq 2\abs{\tau}\\\abs{t_j-\abs{\tau}}> 1}}&\ll \sum_{1\leq n \leq \log_2(\abs{\tau})}\sum_{2^n\leq \abs{t_j-\abs{\tau}}\leq 2^{n+1}}\\
 &\ll    \sum_{1\leq n \leq \log_2(\abs{\tau})}\frac{(2^n)^{\sigma-1}}{(1+\abs{s})^{(k+\sigma-1)}}(\tau^{1/2}2^{n/2}+h^{1/4})\abs{sh}^\e\\
& \ll\frac{1}{(1+\abs{s})^{(\sigma-1)}}(\abs{\tau}^{1/2}\max(1, \abs{\tau}^{\sigma-1/2})+h^{1/4})\abs{sh}^\e
 \end{align}
 and similarly for the continuous contribution. When $\sigma >1/2$, this agrees with the bound stated in the theorem.

 \item  If $n\abs{\tau}< \abs{t}\leq (n+1)\abs{\tau}$ with $n\geq 2$ (i.e. a window of size $\abs{\tau}$ at height  $n\abs{\tau}$,  we see that $$G(s,t)\ll e^{-\pi (n-1)\abs{\tau}}\frac{n^{2\sigma-2}}{(1+\abs{s})^{k}},$$ and we bound the contribution as
\begin{equation*}
    \sum_{n\abs{\tau}< t_j\leq (n+1)\abs{\tau}} \ll e^{-\pi (n-1)\abs{\tau}} n^{2\sigma-2}((n\abs{\tau}^2)^{1/2}+h^{1/4})\abs{hs}^{\e}.
\end{equation*}
We can sum these contributions over $n$ from 2 to infinity and get a bound of $\ll h^{1/4}\abs{s}^{-A}$ for any $A$.
\end{enumerate}

For $\abs{\tau} < 1$ the function $G(s,t)$ decays exponentially in $t$ locally uniformly in $s$ and the sum and integral are again holomorphic.

To see the meromorphic continuation further to the left we argue as follows:
The sum in \eqref{another-expansion} is meromorphic in $\C$ with poles at $s=1/2\pm i t_j-l$, $l=0,1,\ldots$ with corresponding residues
\begin{equation}
  \frac{a_j(h)c_j(f)}{h^{\pm i t_j-l}}\res_{s=1/2\pm it_j-l}G_k(s,t_j).
\end{equation}

We note that for $1/2-l+\e\leq \Re(s)\leq 1/2-l+1-\e$ the above bound on vertical lines is still valid for the sum with the same proof as long as we are bounded away from the poles.

The integral in \eqref{another-expansion} can be written as a line integral along the vertical line $\Re(w)=1/2$:
\begin{equation}\label{integral-expression}
\int_{(1/2)} \varphi(w,h)c_{\frac{w-1/2}{i}}(f)G_k\left(s,\frac{w-1/2}{i}\right)\frac{1}{i}dw.
\end{equation}

This can be continued analytically to $s\in \C$ as follows: By holomorphicity we may deform the integral to the left of two small discs centered at $1/2\pm i\tau$. This continues the integral to a holomorphic function in the union of $\Re(s)>1/2$ with the two small discs. The integral  \eqref{integral-expression} is also holomorphic for $-1/2<\Re(s)<1/2 $ by absolute convergence using the above estimates. For this expression we can now deform the line of integration to the right of two small half-discs at the same heights $\pm \tau$. These two holomorphic functions are now both defined in the two discs. To see how these two functions relate, we take the  difference of the functions and find,  by Cauchy's residue theorem, that their difference  equals
\begin{equation}
\Phi(s)=  \frac{(2\pi)^{(k+1)}\Gamma(2s-1)}{2^s\Gamma(k+s-1)\Gamma(s)}\left(\varphi (w,h)\inprod{y^k\abs{f}^2}{E(z,w)}\Big\lvert_{w=1-s}^{  w=s}\right).
\end{equation}
It follows that the integral
\begin{equation}
 \int_\R \varphi(1/2+it,h)c_{t}(f)G_k(s,t)dt
\end{equation}
can be continued to a meromorphic function in $\Re(s)>-1/2$ that is analytic in $\Re(s)\geq 1/2$ and for $-1/2<\Re(s)<1/2$ can be expressed by
\begin{equation}
 \int_\R \varphi(1/2+it,h)c_{t}(f)G_k(s,t)dt-\Phi(s).
\end{equation}

We may repeat this process to get continuation to $s\in \C.$
\end{proof}
\begin{rem}\label{clarifications}
We note that
\begin{equation}\label{simple-residues}
  \res_{s=1/2\pm it_j-l}G_k(s,t_j)=    \frac{(2\pi)^k(-1)^l}{2^{1/2\pm it_j-l} l!}\frac{\Gamma(\pm 2it_j-l)}{\Gamma(-1/2\pm it_j+k-l)\Gamma(1/2\pm it_j-l)},
\end{equation} which makes the residues in Theorem \ref{props-of-D} explicit.
 We note also that the proof gives also the location of all the poles of $D_f(s,h)$ and allows us to compute the residues at all of them including at the residues of Eisenstein series.

 Concerning the bound on $D_f(s,h)$ it can be improved as follows. Assume that the Satake parameters for $u_j$, say $\alpha_j(p)$, $\beta_j(p)$ satisfy
 \begin{equation}
\abs{\alpha_j(p)}, \abs{\beta_j(p)}\leq p^\theta.
 \end{equation}
We call this bound $H(\theta)$. By the Hecke relations for Hecke eigenvalues it follows that the Hecke eigenvalues for $u_j$ satisfy $\abs{\lambda_j(n)}\leq d(n)n^\theta$. By the work of Kim and Sarnak \cite{KimSarnak:2003a} $H(7/64)$ is true, and the Ramanujan--Petersson conjecture predicts that $H(0)$ is true. Assuming $H(\theta)$ we find that the left-hand side of \eqref{FC-bound} is $O(h^{2\theta}T\Delta(hT)^\e)$. Using this in the proof of Theorem \ref{props-of-D}
we find  that for $1/2+\e\leq \sigma\leq 1+\e$ we have the bound
\begin{equation}
D_f(s,h)\ll_{f,\e} h^{1/2 +\theta-\Re(s)}\abs{s}^{1}\abs{hs}^\e.
\end{equation}
Since for $\sigma\geq 1+\e$ we have  $D_f(s,h)\ll_\e 1$, by \eqref{deligne-bound} or \eqref{rankin-selberg}, the Phragm\'en--Lindel\"of principle gives for $1/2+\e\leq \sigma\leq 1+\e$ the estimate
\begin{equation}\label{D-f-bound}
  D_f(s,h)\ll_{f,\e}(h^{\theta}\abs{s})^{2(1-\sigma)}\abs{hs}^\e .
\end{equation}
\end{rem}

\begin{rem}
Note that the bound in Theorem \ref{props-of-D} and its proof is very similar in spirit to \cite[Thm A.1]{Sarnak:2001a}. The reason we get a  better bound than Sarnak is that we move the line of integration further to the left (there are no small eigenvalues for the full modular group) and we use Good's average bound \eqref{Exp-decay}; Sarnak is using individual bounds \cite[(A18)]{Sarnak:2001a}, which are not as strong. A similar analysis is implicit in \cite[p. 459]{Jutila:1996a}.
\end{rem}

From Theorem \ref{props-of-D} combined with the bound in Remark \ref{clarifications} we can now prove the following bound on the smooth sum defined in \eqref{smooth-sum-def}:
\begin{prop}\label{unconditional-smooth}Assume that $H(\theta)$ is true for some $\theta<1/2$. Then
\begin{equation}
  A_f^W(X,h)=O_{f,W}(h^{\theta+\e}X^{1/2+\e})
\end{equation}
for $h=O(X^{\frac{1}{2(1-\theta)}})$.
\end{prop}
\begin{proof}
By a standard complex contour argument using  Mellin inversion, see e.g. \cite[Appendix A]{Nordentoft:2018}, we find
\begin{equation}\label{weighted-bound}
  \sum_{n\in\N}\lambda_f(n+h)\lambda_f(n)\left(\frac{n}{n+h}\right)^{\frac{k-1}{2}}W((n+h)/X)\ll_{f,W} h^{\theta+\e}X^{1/2+\e}.
\end{equation}
Using the mean value theorem we easily find that
\begin{align}
\label{weight-trivial}\left(\frac{n}{n+h}\right)^{\frac{k-1}{2}}&=1+O((k-1){h}/{n}),\\
  W\left(\frac{n+h}{X}\right) &=W\left(\frac{n+h/2}{X}\right)+O_W\left(\frac{h}{X}\right).
\end{align}
Combining this with \eqref{deligne-bound} or
  \eqref{rankin-selberg} we see that the left-hand side of \eqref{weighted-bound} equals
\begin{equation}
  A_f^W(X,h)+O\left(hX^{\e}\right),
\end{equation}
from which the claim follows easily.
\end{proof}
\begin{rem}
Proposition \ref{unconditional-smooth} should be compared with Blomer's bound \cite[Thm 1.3]{Blomer:2004a} as well as Blomer and Harcos's bound \cite[Eq (5)]{BlomerHarcos:2008a}. These results prove a similar bound. Blomer uses Jutila’s variant of the circle method  combined with the spectral large sieve inequalities of Deshouillers and Iwaniec \cite{DeshouillersIwaniec:1982a}, whereas Blomer and Harcos use the spectral theory on $\hbox{GL}_2$ rather than $\hbox{SL}_2$.
The results of Blomer and Harcos provides bounds which are more uniform in the various parameters. The above argument serves as a simpler proof, which recovers the same strength in the $X$-parameter. Note that for the divisor function  and for $r_2(n)$ (case \ref{divisor-sums} and case \ref{hyp-lattice} in the introduction) we have $\Omega $ results in $X$ for the error term, see \cite{IvicMotohashi:1995} and \cite{PhillipsRudnick:1994a}.  These are, up to logarithms, of the same order as the upper bound in Proposition \ref{unconditional-smooth} as $X$ tends to infinity.

\end{rem}

In order to get good estimates on non-smooth sums we use Jutila's explicit formula, which we now explain. For this it is convenient to use a slightly modified sum namely
\begin{equation}B_f(X,h)=\sum_{n+h\leq X}\lambda_f(n+h)\lambda_f(n)\left(\frac{n}{n+h}\right)^{\frac{k-1}{2}}.\end{equation}
Note that by Deligne's bound \eqref{deligne-bound}
and the trivial estimate \eqref{weight-trivial} we have for  $h\leq X$
\begin{equation}\label{B-to-A}B_f(X,h)=\sum_{n\leq X }\lambda_f(n+h)\lambda_f(n)+O(hX^{\e})=A_f(X,h)+O(hX^\e).\end{equation}
For $l=0,1,2,\ldots$ and $t\neq 0$ we define
\begin{equation}
  \gamma_{t,l}=\frac{(4\pi)^k}{2}\frac{(-1)^l}{l!}\frac{\Gamma(2it-l)}{\Gamma(1/2+it-l+k-1)\Gamma(1/2+it-l+1)},
\end{equation}
and note that for $\abs{t}$ bounded away from zero we have, by Stirling's formula,
\begin{equation}\label{stirling-factors}
  \gamma_{t,l}=\frac{(4\pi)^k}{2}\frac{(-1)^l}{l!\sqrt{2\pi}}\frac{4^{it}}{2^{l+1/2}}\frac{(it)^l}{(it)^{k+1/2}}(1+O(\abs{t}^{-1})).
\end{equation}

\begin{prop}\label{spectral-expansion} Assume that $1\leq h, T\leq X$, and $Th\leq X^{1-\delta}$ for some $\delta>0$. Then there exists a $B>0$ such that
\begin{equation}B_f(X,h)= X^{1/2}\sum_{\substack{t_j\leq T,\, \pm\\ 0\leq l\leq B}}{a_j(h)c_j(f)}\gamma_{\pm t_j, l}(X/h)^{\pm i t_j-l} +O((X/T+X^{1/2})X^\e)\end{equation}
\end{prop}
A similar result is stated in \cite[(1.29)]{Jutila:1996a}. For the reader's convenience we sketch a proof:
\begin{proof}
  Perron's formula, see e.g. \cite[Cor. 5.3]{MontgomeryVaughan:2007a}, gives  for $T\gg 1$
  \begin{align}
    B_f(X,h)=\frac{1}{2\pi i}\int_{1+\e-iT}^{1+\e+iT}2^{s+k-1}D_f(s,h)\frac{X^s}{s}ds+R,
  \end{align}
  where
  \begin{equation}
    R\ll \sum_{X/2<(n+h)<2X}\abs{\lambda_f(n+h)\lambda_f(n)}\min\left(1,\frac{X}{T\abs{X-(n+h)}}\right)+\frac{X}{T}X^\e.
  \end{equation}
  Here we have used that
  \begin{equation}
  \sum_{n=1}^\infty\frac{\abs{\lambda_f(n+h)\lambda_f(n)}}{(n+h)^{1+\e}}\ll_\e 1,
  \end{equation} as follows from the Rankin--Selberg bound \eqref{rankin-selberg} or from Deligne's bound \eqref{deligne-bound}. Using Deligne's bound once more, we can bound the sum in $R$ and find  $R=O(({X}/{T})X^\e)$.

We now move the line of integration to $\Re(s)=1/2+\e$. The contribution from the corresponding horizontal lines can be analyzed as follows. Using Theorem \ref{props-of-D} and the estimate $D_f(s,h)\ll_\e 1$ for $\Re(s)=1+\e$, we note that on both horizontal lines $\Re(s)=1/2+\e, 1+\e$ we have
\begin{equation} \label{convexity-bound}
    D_f(s,h)\frac{X^s}{s}\ll_{f,\e}\left(\frac{X}{\abs{s}}+\frac{X^{1/2}h^{1/4}}{\abs{s}^{1/2}}+X^{1/2}\right)(X\abs{s})^\e.
\end{equation}
Let $g(s)=\left(Xs^{-1}+X^{1/2}h^{1/4}s^{-1/2}+X^{1/2}\right)(Xs)^\e$, where we have chosen the principal powers. It is elementary to verify that for $\Re(s)>\e$ we have
\begin{equation*}
    \abs{g(s)}\gg \left(\frac{X}{\abs{s}}+\frac{X^{1/2}h^{1/4}}{\abs{s}^{1/2}}+X^{1/2}\right)(X\abs{s})^\e.
\end{equation*}
We therefore have $D_f(s,h)/g(s)\ll 1$ on both lines $\Re(s)=1/2+\e, 1+\e$ with a constant not depending on $h, X$ and we may conclude from the Phragm\'en--Lindel\"of principle that  \eqref{convexity-bound} holds for all $1/2+\e\leq \Re(s)\leq +\e$. It follows that the contributions of the horizontal lines in the integral are bounded by a constant times
\begin{equation*}
    \left(\frac{X}{T}+\frac{X^{1/2}h^{1/4}}{T^{1/2}}+X^{1/2}\right)X^\e.
\end{equation*}
Note that since $h\leq X$ the middle term is bounded by $X^{3/4}/T^{1/2}$, which is always smaller than one of the two remaining terms, so it may be ignored.

We now define \begin{equation}
  S_T(s,h)=S_T^d(s,h)+S_T^c(s,h)
\end{equation}
with
\begin{align}
S_T^d(s,h)= & \frac{1}{h^{s-1/2}}\sum_{t_j\leq T} a_j(h) c_j(f)G_k(s,t_j),\\
S_T^c(s,h)=&  \frac{1}{h^{s-1/2}}\frac{1}{4\pi}\int_{\abs{u}\leq T} \varphi(1/2+iu,h)c_{u}(f)G_k(s,u)du.
\end{align}

For $s$ with $\Re(s)=1/2+\e$ and $\abs{\Im{s}}\leq T$ we bound  $D_f(s,h)-S_T(s,h)$ as in the proof of Theorem \ref{props-of-D}. Note that in all remaining contributions coming from \eqref{another-expansion} the bound on $G_k(s,t)$ has a factor of $e^{-\pi(\abs{t}-\abs{\tau})}\leq e^{-\pi(T-\abs{\Im(s)})}e^{-\pi/2(\abs{t}-\abs{\tau})}$. We may pull out the factor $e^{-\pi(T-\abs{\Im(s)})}$ and bound the rest exactly as in Theorem \ref{props-of-D} with the same result and we find under these conditions that
\begin{equation*}
    D_f(s,h)=S_T(s,h)+O_f(e^{-\frac{\pi}{2}(T-\abs{\Im(s)})}(\abs{s}+\abs{s}^{1/2}h^{1/4})\abs{sh}^\e).
\end{equation*}
Inserting this we find
\begin{equation}
  B_f(X)=\frac{1}{2\pi i}\int_{1/2+\e-iT}^{1/2+\e+iT}2^{s+k-1}S_T(s,h)\frac{X^s}{s}ds+O\left(\left(\frac{X}{T}+ X^{1/2}\right)X^\e\right).
\end{equation}
We now again bound all terms in $S_T(s,h)$ exactly as in the proof of Theorem \ref{props-of-D} and find that on $\Re(s)=1/2+\e$ we have
\begin{equation}\label{morebounds}
  S_T^d(s,h), S_T^c(s,h)=O\left((\abs{s}+\abs{s}^{1/2}h^{1/4})\abs{sh}^\e\right).
\end{equation}
We can therefore extend the $s$-integral to the line from $1/2+\e-i(T+1)$ to $1/2+\e+i(T+1)$ at no additional cost; this is useful in order to keep the distance from the poles when we move the line of integration further to the left.

We now analyze the integral in
\begin{equation}
  B_f(X)=\frac{1}{2\pi i}\int_{1/2+\e-i(T+1)}^{1/2+\e+i(T+1)}2^{s+k-1}S_T(s,h)\frac{X^s}{s}ds+O\left(\left(\frac{X}{T}+ X^{1/2}\right)X^\e\right)
\end{equation}
by analyzing the contributions from $S_T^d(s,h)$ and $S_T^c(s,h)$ separately.

To analyze the contribution coming from $S_T^d$ we move the line of integration to $\Re(s)=-B$ picking up poles of $S_T^d(s,h)$ at the poles of the Gamma functions and at 0. The poles of the Gamma function contribute exactly the sum in the theorem as can be seen from \eqref{simple-residues}, so we need to show that for a suitable choice of $B$ all other contributions give admissible errors.

Note that $G_k(t,0)=0$ so $S_T^d(0,h)=0$ and 0 is a removable pole of the integrand. We note that by the proof of Theorem \ref{props-of-D} we have \begin{equation}
    S_T^d(s,h)\ll_{B,f} h^{1/2-\sigma}\frac{\abs{s}^{1/2}+h^{1/4}}{(1+\abs{s})^{\sigma-1}}\abs{sT}^\e, \quad \sigma\leq 1/2+\e.
\end{equation}
We recall that all terms were in absolute value, so a short sum is bounded by a longer one. This implies that the contribution from the horizontal lines can be bounded by
\begin{align*}
\int_{-B\pm i(T+1)}^{1/2+\e\pm i(T+1)}2^{s+k-1}S_T(s,h)\frac{X^s}{s}ds\ll
& h^{1/2}(T^{1/2}+h^{1/4}) X^\e\int_{-B}^{1/2+\e}\left(\frac{X}{hT}\right)^\sigma d\sigma \\
& \ll (X/T+X^{1/2})X^\e,
\end{align*}
which is an admissible error. The contribution from the line $\Re(s)=-B$ is bounded by
\begin{align*}
\int_{-B- i(T+1)}^{-B+ i(T+1)}2^{s+k-1}S_T(s,h)\frac{X^s}{s}ds\ll
& \left(\frac{X}{h}\right)^{-B}h^{1/2}\int_{-(T+1)}^{T+1}\frac{(1+\abs{\tau})^{1/2}+h^{1/4}}{ (1+\abs{\tau})^{-B}} d\tau X^{\e} \\
& \ll \left(\frac{X}{Th}\right)^{-B} h^{1/2}(T^{3/2}+h^{1/4}T)X^{\e}.
\end{align*}
Since $Th\leq X^{1-\delta}$, the constant $B$ may be chosen large enough such that this is also an admissible error. This finishes the analysis of the discrete part of the spectrum.

To analyze the contribution from the continuous spectrum  $S_T^c(s)$ we need, after changing the order of integration, to estimate
\begin{align}\label{does-it-end}
\frac{1}{4\pi}\int_{\abs{t}\leq T}\varphi(1/2+it,h)c_t(f)\frac{1}{2\pi i}\int_{1/2+\e-i(T+1)}^{1/2+\e+i(T+1)}2^{s+k-1}\frac{G_{k}(s,t)}{h^{s-1/2}}\frac{X^s}{s}dsdt .
\end{align}
After moving the line of integration to $-B'$ the inner integral equals \begin{equation}
\sum_{\substack{\pm\\ l\leq B'}}  \gamma_{t,l}\frac{X^{1/2\pm it-l}}{h^{\pm it-l}} + \frac{1}{2\pi i}\int_{-B'-i(T+1)}^{-B'+i(T+1)}2^{s+k-1}\frac{G_{k}(s,t)}{h^{s-1/2}}\frac{X^s}{s}ds+O(\frac{X^{1/2+\e}}{T^{k+1/2}}).
\end{equation}
Bounding using Stirling's formula on all Gamma factors we find that this is bounded by a constant times
\begin{equation}
\frac{X^{1/2}}{(1+\abs{t})^{k+1/2}}\sum_{l\leq B' }\left(\frac{X}{h(1+\abs{t})}\right)^{-l}+\left(\frac{X}{hT}\right)^{-B'}\frac{T^2h^{1/2}}{T^k}+\frac{X^{1/2+\e}}{T^{k+1/2}}.
\end{equation}
 Inserting that back in \eqref{does-it-end} we see that for $B'$ large enough the contribution from  this part of the $u$-integral is $O(X^{1/2+\e})$.  We conclude that the contribution of the continuous spectrum is $O(X^{1/2+\e})$.
\end{proof}

We notice that by \eqref{stirling-factors} the factors in the spectral expansion in Proposition \ref{spectral-expansion} for $\abs{t}$ bounded away from zero  can be bounded as
\begin{equation}\label{coefficient-bound}
  \abs{\gamma_{t,l}}\ll \frac{1}{l!}\frac{\abs{t}^l}{\abs{t}^{k+1/2}}.
\end{equation}
Bounding everything trivially, i.e. using \eqref{Exp-decay} and \eqref{FC-bound}, we find that
\begin{equation}B_f(X,h)=O((X^{1/2}T^{1/2}+ X^{1/2}h^{1/4/T^{1/2}}+X/T+X^{1/2})X^\e).\end{equation}
We choose $T={X^{1/3-\e}}$ to balance the error terms. This gives
\begin{equation}\label{2/3-bound}
  B_f(X,h)=O(X^{2/3+\e})\end{equation}
uniformly for  $h\ll X^{2/3}$. Combined with \eqref{B-to-A}  this recovers \eqref{jutila-bound}.
\subsection{Mean square bounds}

In order to prove Theorem \ref{mean-square-theorem} we first state a general result loosely stating that the mean square of a function with  certain properties is small if and only the function only has few well-spaced points at which the function is large. The proof is a relatively straightforward adaptation of \cite[Sec.4]{IvicMotohashi:1994a}.
\begin{lem}\label{technical-lemma} Consider a measurable function $\abs{B(X,h)}$ defined for $X\geq 2$. Assume that $\abs{B(X,h)}\leq X^A$ for some $A>0$ independent of $h$, and that for $1\leq Y\leq X$
  \begin{equation}
\abs{B(X,h)}^2\ll_\e \frac{1}{Y}\int_{X}^{X+Y}\abs{B(t,h)}^2dt+ Y^2X^\e
\end{equation} uniformly for $h$ in some set $A_{X}$. Let $C_h>0$ be a constant depending on $h$.
Then the following are equivalent:
\begin{enumerate}
  \item For every $\e>0$ we have
  \begin{equation}
    \int_{X}^{2X}\abs{B(t,h)}^2dt\ll_\e C_hX^{2+\e}
\end{equation} uniformly in $h\in A_X$.
  \item \label{second-equivalent} Given $X\leq X_1<X_2<\cdots <X_{R}\leq 2X$ with spacings $X_{r+1}-X_r\geq V\geq X^{1/2+\e}$ and satisfying $\abs{B(X_r,h)}\geq V$. Then
  \begin{equation}R\ll_\e C_h X^{2+\e}V^{-3}
  \end{equation}
  uniformly in $h\in A_X$.
\end{enumerate}
\end{lem}
To see that we may apply Lemma \ref{technical-lemma} to $B_f(X,h)$ we note that by Deligne's bound \eqref{deligne-bound} we have
\begin{equation}
\abs{B_f(X,h)}\leq \sum_{n+h\leq X}{d(n+h)d(n)}\ll X^{1+\e}
\end{equation} uniformly in $h$. Also using Deligne's bound we have for any  with $0\leq t\leq X$  \begin{equation}
 B_f(X+t,h)-B_f(X,h)\ll_\e  \max{(1, \abs{t})}X^{\e}
 \end{equation}
uniformly in $h$. Integrating over $t$ we find that for $1\leq Y\leq X$
\begin{equation}
\frac{1}{Y}\int_{X}^{X+Y}B_f(t,h)dt-B_f(X,h)\ll_\e YX^{\e},
 \end{equation}
and using Cauchy--Schwarz we find
\begin{equation}
\abs{B_f(X,h)}^2\ll_\e \frac{1}{Y}\int_{X}^{X+Y}\abs{B_f(t,h)}^2dt +  Y^2X^\e.
\end{equation}
We conclude that the assumptions of Lemma \ref{technical-lemma} are satisfied.

We now verify that Lemma \ref{technical-lemma} \eqref{second-equivalent} holds:  Let $X\leq X_1<X_2<\cdots <X_{R}\leq 2X$ with spacing $X_{r+1}-X_r\geq V\geq X^{1/2+\e}$ and satisfying $\abs{B_f(X_r,h)}\geq V$. We may assume, without loss of generality, that \begin{equation}V\leq X^{2/3+\e},\end{equation} for if $V\geq X^{2/3+\e}$ it follows from \eqref{2/3-bound} that for $X$ sufficiently large there does not exist any $X\leq X_r\leq 2 X$ satisfying $\abs{B_f(X_r,h)}\geq V$ and in this case the condition is clear.

We assume $R\geq 1$. For each $r$ we choose a complex number $\theta_r$ with norm 1 such that $\theta_r B_f(X_r,h)\geq V$.   Now Proposition \ref{spectral-expansion} gives that when $hT\leq X^{1-\delta}$ we have
\begin{align}
RV &\leq \sum_{\substack{r=1}}^R \theta_r B_f(X_r,h)\\
& =\sum_{\substack{t_j\leq T,\, \pm\\ l\leq B}}{a_j(h)c_j(f)}\gamma_{\pm t_j, l}\sum_{\substack{r=1}}^R\theta_r X_r^{1/2}(X_r/h)^{\pm i t_j-l} +O(R( X/T+X^{1/2})X^{\e/3}).\label{bound-this}
\end{align}
Using dyadic decomposition,  Cauchy--Schwarz, \eqref{coefficient-bound}, and \eqref{Exp-decay}, the sum can be bounded by
\begin{align}
  \sum_{\substack{l\leq B\\\pm}}&\log T\max_{T'\leq T}\abs{\sum_{t_j \sim T'}{a_j(h)c_j(f)}\gamma_{\pm t_j, l}\sum_{\substack{r=1}}^R\theta_r X_r^{1/2}(X_r/h)^{\pm i t_j-l}}\\
  &\leq\sum_{\substack{l\leq B\\\pm}}\log T\max_{T'\leq T}T'^{l-1/2}\left(\sum_{t_j \sim T'}\frac{1}{\cosh(\pi t_j)}\abs{a_j(h)\sum_{\substack{r=1}}^R\theta_r X_r^{1/2}(X_r/h)^{\pm i t_j-l}}^2\right)^{1/2},
\end{align}
where the index $t_j\sim  T$ means that we are summing over $t_j$ satisfying $T\leq t_j\leq 2T$.
In order to estimate this we use a large sieve inequality due to Jutila \cite{Jutila:2000a}:
\begin{thm}\label{jutila-sieve} Let $X\leq X_1<X_2<\cdots <X_{R}\leq 2X$ with spacing $X_{r+1}-X_r\geq V>0$, and let $1\leq T$, $1\leq \Delta\leq T$ and $1\leq N$. Then for any $\e>0$ we have
  \begin{align}\sum_{T\leq t_j\leq T+\Delta} \frac{1}{\cosh(\pi t_j)}&\abs{\sum_{N\leq n\leq 2N}\sum_{r=1}^R b_{n,r} a_j(n)X_r^{\pm it_j}}^2
    \\ &\ll_\e (T+N)(\Delta+\frac{X}{V}\log(\frac{2X}{V}))\norm{b}^2(TN)^\e
  \end{align}
where $b=(b_{n,r})$ is any complex vector and $\norm{b}$ is the standard  $l^2$-norm of $b$.
\end{thm}
\begin{proof} This is \cite[Theorem 4.1]{Jutila:2000a} specialized to $\varphi_n(x,y)$=1, $\psi_n(x,y)=\pm x\log(1+y)/(2\pi)$, $\Psi=\Psi'=T$, $\lambda=\Phi=1$, $y_r=X_r/X-1$, $\delta=V/X$.
\end{proof}

Applying Theorem  \ref{jutila-sieve} we find that the sum in \eqref{bound-this} is bounded by an absolute constant $C_\e$ times
\begin{align}
\sum_{\substack{l\leq B\\\pm}}&\log T\max_{T'\leq T}T'^{l-1/2}\left((T'+h)\left(T'+\frac X V \log(\frac{2X}{V})\right)X^{1-2l}h^{2l}R) X^{\e/4}\right)^{1/2}\\
&\ll \left(1+h\right)^{1/2}\left(T + \frac{X}{V}\right)^{1/2}X^{1/2}R^{1/2}X^{\e/4},
\end{align}
where we are using that $Th/X\ll 1$. This gives
\begin{equation} RV\ll (1+h)^{1/2}\left(T+ \frac{X}{V}\right)^{1/2}X^{1/2}R^{1/2}X^{\e/4} +R(X/T+X^{1/2})X^{\e/3}.\end{equation}
We choose  $T=X^{1+\e/2}/V$ and note that with this choice the last term satisfies
\begin{equation}
   R(X/T+X^{1/2})X^{\e/3}=RV(X^{\frac{\e}{3}-\frac{\e}{2}}+X^{1/2+\e/3}/V) =o(RV)
\end{equation} since $V\geq X^{1/2+\e}$. It follows that
\begin{equation}
    RV\ll (1+h)^{1/2}\left( \frac{X}{V}\right)^{1/2}X^{1/2}R^{1/2}X^{\e/2}.
\end{equation}
Squaring and  dividing by $RV^2$ we find that
\begin{equation}
  R\ll (1+h)\frac{X^2}{V^3}X^{\e}.
\end{equation}
If $h\leq X^{1/2}$ we have $hT/X\leq
{X^{1/2+\e/2}}/{V}\leq X^{-\e/2}$ so all estimates hold uniformly in this range; this requirement comes from $Th\leq X^{1-\delta}$ in Proposition \ref{spectral-expansion}.
  Using Lemma \ref{technical-lemma} we conclude the following result:
  \begin{thm} Let $h\geq 1$. Then
    \begin{equation}
      \left(\frac{1}{X}\int_{1}^{X}\abs{B(t,h)}^2dt\right)^{1/2}\ll_\e h^{1/2}X^{1/2+\e},
    \end{equation}
 uniformly for $h\ll X^{1/2}$.
  \end{thm}

  Using \eqref{B-to-A} we see that this also implies Theorem \ref{mean-square-theorem}. We note that in order to improve on the $h$ dependence with this method it is crucial to improve on the factor $(T+N)$ in Theorem \ref{jutila-sieve}.

  \section{The variance over  a Hecke basis}
  We now go back to the smooth shifted convolution sums $A_f^{W}(X,h)$, and investigate the variance over an orthonormal basis of Hecke eigenforms $H_k$. More precisely we want to understand \begin{equation}
\frac{2\pi ^2}{k-1}\sum_{f\in H_k}  \frac{A_f^{W_1}(h_1,X)\overline{A_f^{W_2}(h_2,X)}}{L(1, \sym^2 f)}.\end{equation}

 In order to better describe the dependence on $W$, and $h$ we use the Sobolev norms
 \begin{align}
 \begin{split}
 \label{sobolev-norms}\norm{W}_{l,p}^p=&\sum_{0\leq i\leq l}\norm{\frac{d^i}{dy^i}W}_p^p,\\ \norm{W}_{l,\infty}=&\sum_{0\leq i\leq l}\norm{\frac{d^i}{dy^i}W}_\infty.
 \end{split}
 \end{align}

For all compactly supported functions $W$ we assume  $\supp{W_i}\subseteq [a_W,A_W]$, and that $A_W\geq 1$, and for two shifts $h_1$, $h_2$ we denote $\norm{h}_\infty=\max(h_1,h_2)$.
  \subsection{Small range asymptotics}
  We investigate the case where the size of the range of the sum in $A_f^{W}(X,h)$ is small compared  to the size of the additional average i.e. $\dim(S_k)\sim k/12$. Specifically we investigate the range $X\leq k^{1/2-\e}$.

  The main tool in understanding this is the Petersson formula, which states that
  \begin{align} \frac{2\pi ^2}{k-1}& \sum_{f\in H_k} \frac{\lambda_f(n_1)\lambda_f(n_2)}{L(1, \sym^2(f))}\\& = \delta_{n_1,n_2}+2\pi (-1)^{k/2} \sum_{\substack{c\geq 1}} \frac{S(n_1,n_2;c)}{c}J_{k-1}\left( \frac{4\pi \sqrt{n_1n_2}}{c}\right), \label{Petersson}\end{align}
   see e.g. \cite[p. 776]{LuoSarnak:2004a}. When $X\leq k^{1/2-\e}$ we will see that we can make good use  of the following bound on the $J$-Bessel function
     \begin{equation} \label{good-bound-Bessel} J_{k-1}(x)\ll \left( \frac{ex}{2k}\right)^{k-1}, \end{equation}
  see e.g. \cite[p. 233]{LiuMasri:2014}. In fact the decay \eqref{good-bound-Bessel} will imply that the contribution from the sum over $c$ becomes essentially neglegible for this  problem. This means that we are reduced to studying the diagonal terms. In our treatment we take extra care to get an explicit error-term.

  In order  to use the Petersson formula we note that, since the normalized Hecke eigenvalues satisfy the Hecke relations
 \begin{equation}\label{hecke-relations}
     \lambda_f(n)\lambda_f(m)=\sum_{d\vert(m,n)}\lambda_{f}\left(\frac{mn}{d^2}\right),
 \end{equation}
 see \cite[(6.38)]{Iwaniec:1997a},
we may rewrite
  \begin{align}
      A_f^{W}(X,h)&= \sum_{n\in \N}\sum_{d\vert (n,n+h)}\lambda(n(n+h)/d^2) W((n+h/2)/X)\\
\label{after-Hecke-relations}      &=\sum_{d\mid h} \sum_{r\in \N}\lambda_f(r(r+d))W\left(\frac{\frac{h}{d}(r+d/2)}{X}\right).
  \end{align}

  To state our theorem in this case we recall from the introduction that for functions $W_1, W_2:(0,\infty)\rightarrow \R$ and $h_1,h_2\in \N$ we define
 \begin{align}\label{BLF} B_{h_1,h_2}(W_1, W_2)=  \tau_1((h_1,h_2))\int_{0}^\infty W_1(h_1 y) \overline{W_2(h_2y)}dy. \end{align}

\begin{thm} \label{scs}
Let $W_1, W_2:(0,\infty)\rightarrow \R$ be smooth functions with compact support. Then for $A_{W_1}A_{W_2}X\ll k^{1/2-\e}$, $h_i<2a_{W_i}X$,
  \begin{align}
  \label{shiftedconv} \frac{2\pi ^2}{k-1}\sum_{f\in H_k}  \frac{A_f^{W_1}(h_1,X)\overline{A_f^{W_2}(h_2,X)}}{L(1, \sym^2 f)} =  B_{h_1,h_2}(W_1,W_2) X+ O_{W_i,h_i}(1).
  \end{align}
The implied constant is $\ll_\e \norm{h}_\infty^{\e}(\norm{W_1(h_1\cdot)W_2(h_2\cdot)}_{2,1} +\norm{W_1}_\infty\norm{W_2}_\infty)$ for any $\e>0$.
  \end{thm}
  \begin{proof}
Using \eqref{after-Hecke-relations} and the Petersson formula \eqref{Petersson}  we find
  \begin{align}
\label{some-expression}  &\frac{2\pi ^2}{k-1}\sum_{f\in H_k} \frac{A_f^{W_1}(h_1,X)\overline{A_f^{W_2}(h_2,X)}}{L(1, \sym^2 f)}
   \\
  &=\sum_{\substack{d_1 \mid h_1 \\ d_2 \mid h_2}} \sum_{r_1,r_2\in\N} \delta_{r_1(r_1+d_1)=r_2(r_2+d_2)} W_1\left( \frac{h_1(r_1+d_1/2)}{d_1 X} \right)\overline{W_2\left( \frac{h_2(r_2+d_2/2)}{d_2X} \right)}\\
  +&2\pi (-1)^{k/2} \sum_{\substack{d_1 \mid  h_1 \\ d_2 \mid h_2}} \sum_{r_1,r_2\in \N} W_1\left( \frac{h_1(r_1+d_1/2)}{d_1 X} \right)\overline{W_2\left( \frac{h_2(r_2+d_2/2)}{d_2X} \right)}\\
  &\times \sum_{\substack{c\geq 1}} \frac{S(r_1(r_1+d_1),r_2(r_2+d_2);c)}{c}J_{k-1}\left( \frac{4\pi \sqrt{r_1(r_1+d_1)r_2(r_2+d_2)}}{c}\right). \end{align}
We refer to the line with the Kronecker delta as the diagonal term, and the rest as the off-diagonal term.

To handle the diagonal term, we observe that for fixed positive  $d_1\neq d_2$ the equation \begin{equation}\label{finitely-many} r_1(r_1+d_1)=r_2(r_2+d_2)\end{equation}
has only finitely many positive integer solution $(r_1,r_2)\in \N^2$. To see this we rewrite the equation as $(2r_1+d_1)^2-(2r_2+d_2)^2=d_1^2-d_2^2$.
Factoring the left-hand side as $(2r_1+d_1+2r_2+d_2)(2r_1+d_1-2r_2-d_2)$ we see that any solution gives a factorization of  $d_1^2-d_2^2$, and that any factorization of $d_1^2-d_2^2$ comes from at most one solution.
This shows that there are at most $d(d_1^2-d_2^2)$ solutions to  \eqref{finitely-many}  with $d_1\neq d_2$, where $d(n)$ denotes the number of divisors of $n$; indeed we see that the total contribution from these terms is $O(\norm{h}_\infty^{ \e}\norm{W_1}_\infty\norm{W_2}_\infty)$.

For the remaining terms, i.e. $d_1=d_2=d$, $r_1=r_2=r$  we apply first Poisson summation in the $r$-variable  and use the fact that the Fourier transform of the function $y\mapsto W_1\left( \frac{h_1(y+d/2)}{dX} \right)\overline{W_2\left( \frac{h_2(y+
d/2)}{dX}\right)}$ at $r$ is bounded by an absolute constant times
\begin{equation} \label{coffeetime}
      \abs{r}^{-n}(dX)^{-n+1}\norm{W_1(h_1\cdot)W_2(h_2\cdot)}_{n,1}.\end{equation}
This follows from repeated integration by parts.

By the assumption on $h_i$ we now see that
  \begin{align}
  & \sum_{\substack{d_1 \mid h_1 \\ d_2 \mid h_2\\ d_1=d_2}} \sum_{r\in\N} W_1\left( \frac{h_1(r+d_1/2)}{d_1X} \right)\overline{W_2\left( \frac{h_2(r+
  d_2/2)}{d_2X}\right)}\\ \allowdisplaybreaks
 \label{extendtoZ} & =\sum_{d\vert(h_1,h_2)} \sum_{r\in\Z} W_1\left( \frac{h_1(r+d/2)}{dX} \right)\overline{W_2\left( \frac{h_2(r+
  d/2)}{dX}\right)}\\
  &= \sum_{d\mid (h_1,h_2)} \sum_{r\in \Z}\int_{-\infty}^\infty W_1\left(\frac{h_1 (y+d/2)}{dX}\right)\overline{W_2\left(\frac{h_2 (y+d/2))}{dX}\right)}e(-ry)dy\\
  &= \tau_1((h_1,h_2))\int_{-\infty}^\infty W_1(h_1 y) \overline{W_2(h_2y)}dyX +O(\norm{h}_\infty^{\e}\norm{W_1(h_1\cdot)W_2(h_2\cdot)}_{2,1} X^{-1} ),
  \end{align}
where, in the first equality we have extended  the $r$-sum to all of $r$ trivially as    all added terms are zero, since $h_i<2a_{W_i}X$. In the second equality we  use Poisson summation and finally evaluate the $r=0$ term and bound the rest using  \eqref{coffeetime} with $n=2$.

Using  the bound \eqref{good-bound-Bessel} on the $J$-Bessel function, we may bound the off-diagonal term in the expression for \eqref{some-expression}
  by an absolute constant times $\norm{W_1}_\infty \norm{W_2}_\infty$ times
  \begin{align*}
   \sum_{d_i\vert h_i}\left(\sum_{r_i+\frac{d_i}{2}\leq A_{W_i}d_iX/h_i} \left(\frac{e 4\pi \sqrt{r_1(r_1+d_1) r_2(r_2+d_2)}}{2k}\right)^{k-1}\right) \left(\sum_{\substack{c\geq 1}} c^{-(k-1)}   \right),
  \end{align*}
using the trivial bound for Kloosterman sums and the sup norm for the test functions. Note that since $r_i(r_i+d_i)=(r_i+d_i/2)^2-d_i^2/4$ the square root is $\ll A_{W_1}A_{W_2}X^2$. Since we have  assumed that $A_{W_i}X\leq A_{W_1}A_{W_2} X\ll k^{1/2-\e}$, we find
\begin{equation}   \left(\frac{e4\pi \sqrt{r_1(r_1+d_1) r_2(r_2(r_2+d_2)})}{2k}\right)^{k-1} \ll k^{-\e k},\end{equation}
and, therefore, the off-diagonal term decay exponentially in $k$ and only polynomially in $X$. This finishes the proof, as it shows that the contribution of the off-diagonal term is $O(\norm{h}_\infty^\e\norm{W_1}_\infty\norm{W_2}_\infty k^{-A})$ for any $A>0$.
  \end{proof}

  \begin{rem} Note that in Theorem \ref{scs} the leading term is coming from the diagonal term and the off-diagonal term is essentially negligible because of the exponential decay of the $J$-Bessel function when $X\ll k^{1/2-\e}$. If $X\gg k^{1/2-\e}$  such decay is not available, and we do not know how to estimate the sum. However, if we make an additional average over $k$, then the off-diagonal may be analyzed, and it is again possible to prove an asymptotic formula. See \cite[Thm 2.1]{NordentoftPetridisRisager:2020a} for details.

  \end{rem}

  \begin{rem}\label{remove-condition}
  We can remove the condition $h_i<2a_{W_i}X$ in Theorem \ref{scs} as follows. The only place this restriction is used is to extend the summation in \eqref{extendtoZ} to $\Z$ instead of $\N$ since in this case there are no $r\leq 0$ with $a_{W_i}({d_i}/{h_i})X\leq r_i+ {d_i}/{2}$. Without $h_i<2a_{W_i}X$ there can be at most $d_i$ such non-positive $r$ as $a_{W_i}>0$ so we may omit the restriction on $h_i$ at the expense of an extra error term of size $O(\norm{h}_\infty
 ^{1+\e}\norm{W_1}_\infty \norm{W_2}_\infty$).
  \end{rem}
  \begin{rem}
  We can also prove Theorem \ref{scs} when $W_i(y)=1_{y\leq 1}$ are sharp cut-off functions, and we get an error term  $O(\norm{h}_\infty^{1+\e})$. In this case the Poisson summation argument is replaced by an elementary count. An interpretation of Theorem \ref{scs} with sharp cut-offs is that we obtain the conjectured square-root cancellation in the $X$-aspect for the shifted convolution problem  on average.
  \end{rem}

\bibliographystyle{amsplain}
\providecommand{\bysame}{\leavevmode\hbox to3em{\hrulefill}\thinspace}
\providecommand{\MR}{\relax\ifhmode\unskip\space\fi MR }
\providecommand{\MRhref}[2]{%
  \href{http://www.ams.org/mathscinet-getitem?mr=#1}{#2}
}
\providecommand{\href}[2]{#2}

\section*{Data Availability}
Data sharing not applicable to this article as no datasets were generated or analysed during the current study.

\end{document}